\documentclass[12pt,english,a4paper]{smfart}

\usepackage[T1]{fontenc}
\usepackage{lmodern}
\usepackage{smfthm}
\usepackage[headings]{fullpage}
\usepackage{amssymb}

\newcommand{\Gal}{\operatorname{Gal}}
\newcommand{\Qp}{\mathbf{Q}_p}
\newcommand{\Zp}{\mathbf{Z}_p}
\newcommand{\Fp}{\mathbf{F}_p}
\newcommand{\ZZ}{\mathbf{Z}}
\newcommand{\RR}{\mathbf{R}}
\newcommand{\eps}{\varepsilon}

\newcommand{\HH}{\mathcal{H}}
\newcommand{\WW}{\mathcal{W}}
\newcommand{\Aut}{\operatorname{Aut}}
\newcommand{\Mat}{\operatorname{Mat}}

\newcommand{\Nm}{\operatorname{N}}
\newcommand{\sh}{\operatorname{sh}}
\newcommand{\orb}{\operatorname{orb}}
\newcommand{\cont}{\operatorname{cont}}
\newcommand{\cyc}{\operatorname{cyc}}
\newcommand{\fin}{\operatorname{fin}}

\newcommand{\val}{\operatorname{val}}
\newcommand{\hyphen}{\!\operatorname{-}\!}
\newcommand{\vp}{\val_p}
\newcommand{\vm}{\val_M}
\newcommand{\vx}{\val_X}
\newcommand{\bigO}{\operatorname{O}}
\newcommand{\dcroc}[1]{[\![ #1 ]\!]}
\newcommand{\dpar}[1]{(\!( #1 )\!)}
\newcommand{\efont}{\mathbf{E}}
\newcommand{\kf}{E}
\newcommand{\dfont}{\mathbf{D}}
\newcommand{\et}{\widetilde{\mathbf{E}}}
\newcommand{\M}{\operatorname{M}}
\newcommand{\GL}{\operatorname{GL}}
\newcommand{\Id}{\operatorname{Id}}
\newcommand{\BB}{\mathrm{B}_2(\Qp)}
\newcommand{\tco}{{}^t\!\operatorname{co}}
\newcommand{\Gm}{\mathbf{G}_\mathrm{m}}
\renewcommand{\geq}{\geqslant}
\renewcommand{\leq}{\leqslant} 
\renewcommand{\phi}{\varphi} 

\NumberTheoremsIn{section}

\author{Laurent Berger}
\address{UMPA de l'\'ENS de Lyon \\
UMR 5669 du CNRS}
\email{laurent.berger@ens-lyon.fr}
\urladdr{perso.ens-lyon.fr/laurent.berger/}

\author{Sandra Rozensztajn}
\address{UMPA de l'\'ENS de Lyon \\
UMR 5669 du CNRS}
\email{sandra.rozensztajn@ens-lyon.fr}
\urladdr{perso.ens-lyon.fr/sandra.rozensztajn/}

\title{Decompletion of cyclotomic perfectoid fields in positive characteristic}

\date{\today}

\begin{document}

\begin{abstract}
Let $\kf$ be a field of characteristic $p$. The group $\Zp^\times$ acts on $\kf \dpar{X}$ by $a \cdot f(X) = f((1+X)^a-1)$. This action extends to the $X$-adic completion $\et$ of $\cup_{n \geq 0} \kf \dpar{X^{1/p^n}}$. We show how to recover $\kf \dpar{X}$ from the valued $\kf$-vector space $\et$ endowed with its action of $\Zp^\times$. To do this, we introduce the notion of super-H\"older vector in certain $\kf$-linear representations of $\Zp$. This is a characteristic $p$ analogue of the notion of locally analytic vector in $p$-adic Banach representations of $p$-adic Lie groups.
\end{abstract}

\subjclass{11S; 12J; 13J; 22E}


\maketitle

\setcounter{tocdepth}{2}
\tableofcontents

\setlength{\baselineskip}{18pt}
\section*{Introduction}

Let $p$ be a prime number, and let $\kf$ be a field of characteristic
$p$. Let $\efont = \kf \dpar{X}$, and let $\et$ be the $X$-adic
completion of $\cup_{n \geq 0} \kf \dpar{X^{1/p^n}}$. Note that if $\kf$
is perfect, the field $\et$ is perfectoid. The group $\Zp^\times$ acts on
$\efont$ by $\left(a \cdot f\right)(X) = f((1+X)^a-1)$. This action extends to
$\cup_{n \geq 0} \kf \dpar{X^{1/p^n}}$ by $\left(a \cdot f\right)(X^{1/p^n}) =
f((1+X^{1/p^n})^a-1)$, and by continuity to $\et$. The question that
motivated this paper is the following.

\begin{enonce*}{Question} 
Can we recover $\cup_{n \geq 0} \kf \dpar{X^{1/p^n}}$ or even $\kf \dpar{X}$ from the data of the valued $E$-vector space $\et$ endowed with the action of $\Zp^\times$?
\end{enonce*}

In characteristic zero, it is possible to answer an analogous question by using Schneider and Teitelbaum's theory of locally analytic vectors in $p$-adic Banach representations of $p$-adic Lie groups. For characteristic $p$ representations, there is no such theory. One of the main contributions of this article is to introduce a characteristic $p$ analogue of locally analytic functions and vectors. 

Let $M$ be an $\kf$-vector space, endowed with a valuation $\vm$ such that $\vm(xm) = \vm(m)$ if $x \in \kf^\times$. We assume that $M$ is separated and complete for the $\vm$-adic topology. For example, we will consider $M= \efont$ or $\et$ with the $X$-adic valuation. We say that a function $f :\Zp \to M$ is super-H\"older if there exist constants $\lambda,\mu \in \RR$ such that $\vm(f(x)-f(y)) \geq p^\lambda \cdot p^i +\mu$ whenever $\vp(x-y) \geq i$, for all $x,y \in \Zp$ and $i \geq 0$. These super-H\"older functions are the characteristic $p$ analogue of locally analytic functions $\Zp \to \Qp$. We prove an analogue of Mahler's theorem for continuous functions $f : \Zp \to M$, and give a characterization of super-H\"older functions in terms of their Mahler expansions. This is a characteristic $p$ analogue of a theorem of Amice.

Assume now that $\Gamma$ is a group that is isomorphic to $\Zp$ via a coordinate map $c$, and that $M$ is endowed with an $\kf$-linear action of $\Gamma$ by isometries. We say that $m \in M$ is a super-H\"older vector if the orbit map $z \mapsto c^{-1}(z) \cdot m$ is a super-H\"older function $\Zp \to M$. This definition is a characteristic $p$ analogue of the notion of locally analytic vector of a $p$-adic Banach representation of a $p$-adic Lie group. We let $M^{\Gamma\hyphen\sh,\lambda}$ denote the space of super-H\"older vectors for a given constant $\lambda$ as in the definition above. We also let $M^{\sh}$ denote the set of super-H\"older vectors in $M$. Our main result is a complete answer to the question above. Consider $M=\et$, endowed with the action of $\Gamma=1+p^k \Zp$ for $k \geq 1$ (or $k \geq 2$ if $p=2$).

\begin{enonce*}{Theorem} 
For all $n \geq 0$, we have $\et^{(1+p^k \Zp) \hyphen\sh, k-n} = \kf \dpar{X^{1/p^n}}$. 

In particular, $\et^{\sh} = \cup_{n \geq 0} \kf \dpar{X^{1/p^n}}$.
\end{enonce*}

The main ingredients of the proof of this theorem are some simple computations in $\kf \dcroc{X}$, as well as Colmez' analogue of Tate traces for $\et$.

We give several applications of our main result. First, we compute the perfectoid commutant of $\Aut(\Gm)$, namely the set of $u \in \et^{\vx >0}$ such that $u \circ \gamma_a = \gamma_a \circ u$ for all $a \in \Zp^\times$, where $\gamma_a(X) = (1+X)^a-1$. Using our main theorem, and a result of Lubin-Sarkis on the classical commutant of $\Aut(\Gm)$, we prove  that such a $u$ is of the form $\gamma_b(X^{p^n})$ for some $b \in \Zp^\times$ and $n \in \ZZ$. Next we study $(\phi,\Gamma)$-modules over $\efont$. We prove that the action of $\Gamma$ on a $(\phi,\Gamma)$-module $\dfont$ is always super-H\"older, and deduce that $(\et \otimes_{\efont} \dfont)^{\sh} = (\cup_{n \geq 0} \kf \dpar{X^{1/p^n}}) \otimes_{\efont} \dfont$. This allows us to extend our computation of super-H\"older vectors to the finite extensions of $\Fp \dpar{X}$ provided by Fontaine and Wintenberger's theory of the field of norms. We finish this article with a computation that suggests that the theory of super-H\"older vectors could have some applications to the $p$-adic local Langlands correspondence.

\vspace{\baselineskip}\noindent\textbf{Acknowledgements.}
We thank Juan Esteban Rodr\'{\i}guez Camargo for asking LB the question that motivated this paper, as well as Christophe Breuil, Daniel Gulotta, Gal Porat and the referee for their comments and questions.

\section{Super-H\"older functions and vectors}
\label{secshf}

In this section, we define super-H\"older functions $\Zp \to M$ and super-H\"older vectors in $M$ when $M$ is a representation of a group isomorphic to $\Zp$. We prove an analogue of Mahler's theorem for continuous functions $\Zp \to M$, and give a characterization of super-H\"older functions in terms of their Mahler expansions.

\subsection{Super-H\"older functions}
\label{subshf}

We keep the notation of the introduction. Let $M$ be an $\kf$-vector space, endowed with a valuation $\vm$ such that $\vm(xm) = \vm(m)$ if $x \in \kf^\times$. We assume that $M$ is separated and complete for the $\vm$-adic topology. For example, we will consider $M= \kf \dcroc{X}$ with the $X$-adic valuation.

Let $C^0(\Zp,M)$ denote the space of continuous functions $f : \Zp \to M$. 

\begin{defi}
\label{defsh}
We say that $f : \Zp \to M$ is super-H\"older if there exist constants $\lambda,\mu \in \RR$ such that $\vm(f(x)-f(y)) \geq p^\lambda \cdot p^i +\mu$ whenever $\vp(x-y) \geq i$, for all $x,y \in \Zp$ and $i \geq 0$.
\end{defi}

We let $\HH^{\lambda,\mu}(\Zp,M)$ denote the space of functions such that $\vm(f(x)-f(y)) \geq p^\lambda \cdot p^i +\mu$ whenever $\vp(x-y) \geq i$, for all $x,y \in \Zp$ and $i \geq 0$, and $\HH^\lambda(\Zp,M) = \cup_{\mu \in \RR} \HH^{\lambda,\mu}(\Zp,M)$ and $\HH(\Zp,M) = \cup_{\lambda \in \RR} \HH^\lambda(\Zp,M)$. 

For example, if $M = \kf \dcroc{X}$ with $\vm=\vx$, then  $[a \mapsto (1+X)^a] \in\HH^{0,0}(\Zp,M)$. Indeed, $(1+X)^a - (1+X)^{a + p^i b} = (1+X)^a (1-(1+X^{p^i})^b) \in X^{p^i} \kf \dcroc{X}$ if $i \geq 0$.

\begin{rema}
\label{hcdcl}
The space $\HH^{\lambda,\mu}(\Zp,M)$ is closed in $C^0(\Zp,M)$.
\end{rema}

\begin{rema}
\label{isom}
If $\alpha : \Zp \to \Zp$ is an isometry, then $f : \Zp \to M$ belongs to $\HH^{\lambda,\mu}(\Zp,M)$ if and only if $f \circ \alpha \in \HH^{\lambda,\mu}(\Zp,M)$
\end{rema}

\begin{prop}
\label{hrng}
Suppose that $M$ is a ring, and that $\vm(mm') \geq \vm(m) + \vm(m')$ for all $m,m' \in M$. If $c \in \RR$, let $M_c= M^{\vm \geq c}$. 

\begin{enumerate}
\item If $f \in \HH^{\lambda,\mu}(\Zp,M_c)$ and $g \in \HH^{\lambda,\nu}(\Zp,M_d)$, and $\xi=\min(\mu+d,\nu+c)$, then $fg \in \HH^{\lambda,\xi}(\Zp,M_{c+d})$.
\item If $\lambda,\mu \in \RR$, then $\HH^{\lambda,\mu}(\Zp,M_0)$ is a subring of $C^0(\Zp,M)$.
\item If $\lambda \in \RR$, then $\HH^{\lambda}(\Zp,M)$ is a subring of $C^0(\Zp,M)$
\item
 If $d \geq 1$, we see $\GL_d(M)$ as a subset of the valued $E$-vector space $\M_d(M)$. If $\lambda,\nu \in \RR$ and $Q \in \HH^{\lambda}(\Zp,\GL_d(M))$ are such that $\vm(\det Q(x)) \leq \nu$ for all $x \in \Zp$, then $Q^{-1} \in \HH^{\lambda}(\Zp,\GL_d(M))$.
\end{enumerate}
\end{prop}

\begin{proof}
Items (2) and (3) follow from item (1), which we now prove.
If $x,y \in \Zp$, then
\[ (fg)(x) - (fg)(y) = (f(x)-f(y))g(x) + (g(x)-g(y))f(y),\]
which implies the claim. We now prove (4). If $d=1$, then 
\[ Q^{-1}(y) - Q^{-1}(x) = \frac{Q(x)-Q(y)}{Q(x)Q(y)}, \]
which implies the claim. If $d \geq 1$, we can write $Q^{-1} = \tco(Q) \cdot \det(Q)^{-1}$, and the claim results from (3), and (4) applied to $d=1$.
\end{proof}

\begin{rema}
\label{notting}
Take $u \in X + X^2 \kf \dcroc{X}$, and let $u^{\circ n}$ be $u$ composed with itself $n$ times. Sen's theorem (\cite{S69}, theorem 1) implies that $\vx(u^{\circ p^k}(X)-X) \geq p^k$ if $k \geq 0$, so that $\vx(u^{\circ x}-u^{\circ y}) \geq p^i$ if $\vp(x-y) \geq i$. This implies that the map $\ZZ_{\geq 0} \to X + X^2 \kf \dcroc{X}$, given by $n \mapsto u^{\circ n}$, extends to a super-H\"older function on $\Zp$.
\end{rema}

\subsection{Super-H\"older vectors}
\label{secvec}

We now assume that $M$ is endowed with an $\kf$-linear action by isometries of a group $\Gamma$, where $\Gamma$ is isomorphic to $\Zp$, via a coordinate map $c$. If $m \in M$, let $\orb_m : \Gamma \to M$ denote the function defined by $\orb_m(a) = a \cdot m$, so that $\orb_m \circ c^{-1}$ is a function $\Zp \to M$. 

\begin{defi}
\label{defshvec}
Let $M^{\Gamma\hyphen\sh,\lambda,\mu}$ denote the set of $m \in M$ such that $\orb_m  \circ c^{-1} \in \HH^{\lambda,\mu}(\Zp,M)$, and let $M^{\Gamma\hyphen\sh,\lambda}$ and $M^{\Gamma\hyphen\sh}$ be the corresponding sub-$\kf$-vector spaces of $M$.
\end{defi}

This definition should be seen as a characteristic $p$ analogue of the locally analytic vectors of a Banach representation of a $p$-adic Lie group, as defined in \S 7 of \cite{ST03}. The requirement that $\Gamma$ acts by isometries is the analogue of the condition that the norm be invariant.

\begin{rema}
\label{mcont}
We assume that $\Gamma$ acts by isometries on $M$, but not that $\Gamma$ acts continuously on $M$, namely that $\Gamma \times M \to M$ is continuous. However, let $M^{\cont}$ denote the set of $m \in M$ such that $\orb_m \circ c^{-1} : \Zp \to M$ is continuous. It is easy to see that $M^{\cont}$ is a closed sub-$\kf$-vector space of $M$, and that $\Gamma \times M^{\cont} \to M^{\cont}$ is continuous (compare with \S 3 of \cite{E17}). We then have $M^{\sh} \subset M^{\cont}$.
\end{rema}

\begin{lemm}
\label{mshalt}
We have $m \in M^{\Gamma\hyphen\sh,\lambda,\mu}$ if and only if $\vm(g \cdot m - m) \geq p^\lambda \cdot p^i +\mu$ for all $g \in \Gamma$ such that $c(g) \in p^i \Zp$.
\end{lemm}

\begin{proof}
Since $\Gamma$ acts by isometries, we have $\vm(hg \cdot m - h \cdot m)= \vm (g \cdot m - m)$ for all $g,h \in \Gamma$.
\end{proof}

\begin{lemm}
\label{cplmsh}
The space $M^{\Gamma \hyphen \sh,\lambda,\mu}$ is a closed sub-$\kf$-vector space of $M$.
\end{lemm}

\begin{lemm}
\label{shsbgr}
If $k \geq 0$ and $\Gamma' = c^{-1}(p^k \Zp)$, then $g \mapsto c(g)/p^k$ is a coordinate on $\Gamma'$, and $M^{\Gamma \hyphen \sh,\lambda} = M^{\Gamma'\hyphen \sh,\lambda+k}$.
\end{lemm}

\begin{proof}
It is clear that $M^{\Gamma \hyphen \sh,\lambda} \subset M^{\Gamma' \hyphen \sh,\lambda+k}$. Conversely, let $C= \{1,\hdots,p^k-1\}$. If $m \in M^{\Gamma' \hyphen \sh,\lambda+k,\mu}$,  let $\nu = \min_{c(h) \in C} \vm(h \cdot m - m)$. If $g \in \Gamma \setminus \Gamma'$, we can write $g=g_k h$ with $c(h) \in C$ and $g_k \in \Gamma'$. We then have $g \cdot m - m = ( g_k \cdot h \cdot m - g_k \cdot m ) + (g_k \cdot m - m)$ so that $\vm(g \cdot m-m) \geq \min(\mu,\nu)$.

This implies that $m \in M^{\Gamma \hyphen \sh,\lambda,\mu'}$ with $\mu' = \min(\mu,\nu)-p^{k+\lambda}$.
\end{proof}

In particular, the space $M^{\Gamma'\hyphen\sh}$ does not depend on the choice of open subgroup $\Gamma' \subset \Gamma$, and we denote it by $M^{\sh}$.

\begin{prop}
\label{propmsh}
Suppose that $M$ is a ring, and that $g(mm') = g(m)g(m')$ and $\vm(mm') \geq \vm(m) + \vm(m')$ for all $m,m' \in M$ and $g \in \Gamma$.
\begin{enumerate}
\item If $v \in \RR$ and $m,m' \in M^{\Gamma \hyphen \sh,\lambda,\mu} \cap M^{\vm \geq v}$, then $m \cdot m' \in M^{\Gamma \hyphen \sh,\lambda,\mu+v}$;
\item If $m \in M^{\Gamma \hyphen \sh,\lambda,\mu} \cap M^\times$, then $1/m \in M^{\Gamma \hyphen \sh,\lambda,\mu-2 \vm(m)}$.
\end{enumerate}
\end{prop}

\begin{proof}
Item (1) follows from prop \ref{hrng}
and lemma \ref{mshalt}. Item (2) follows from 
\[ g\left(\frac 1m\right) - \frac 1m = \frac{m-g(m)}{g(m) m}. \]
\end{proof}

\begin{rema}
\label{jerclocan}
One can extend the definition of super-H\"older vectors to the setting of a $p$-adic Lie group $G$ acting by isometries on a valued $\kf$-vector space $M$ as follows (the details are in our paper \emph{Super-H\"older vectors and the field of norms}). Let $P$ be a nice enough open pro-$p$ subgroup of $G$. We say that $m \in M$ is super-H\"older if and only if there exists $\lambda,\mu \in \RR$ and $e>0$ such that $\vm(g\cdot m - m) \geq p^{\lambda + ei} + \mu$ whenever $g \in P^{p^i}$, for all $i \geq 0$. Juan Esteban Rodr\'{\i}guez Camargo pointed out to us that there is a similar purely metric characterization of locally analytic vectors for a $p$-adic Lie group acting on a Banach space.
\end{rema}

\subsection{Mahler's theorem}
\label{secmahl}

In this section, we prove a characteristic $p$ analogue of Mahler's theorem for continuous functions $\Zp \to \Qp$. We then give a characterization of super-H\"older functions in terms of their Mahler expansions. If $z \in \Zp$ and $n \geq 0$, then $\binom{z}{n} \in \Zp$ and we still denote by $\binom{z}{n}$ its image in $\Fp$.

\begin{theo}
\label{mahlthm}
If $\{m_n\}_{n \geq 0}$ is a sequence of $M$ such that $m_n \to 0$,  the function $f : \Zp \to M$ given by $f(z) = \sum_{n \geq 0} \binom{z}{n} m_n$ belongs to $C^0(\Zp,M)$. We  have $m_n = (-1)^n \sum_{i=0}^n(-1)^i\binom{n}{i} f(i)$ and $\inf_{z \in \Zp} \vm(f(z)) =\inf_{n \geq 0} \vm(m_n)$.

Conversely, if $f \in C^0(\Zp,M)$,  there exists a unique sequence $\{m_n(f)\}_{n \geq 0}$ such that $m_n(f) \to 0$ and such that $f(z) = \sum_{n \geq 0} \binom{z}{n} m_n(f)$.
\end{theo}

\begin{proof}
Our proof follows Bojanic's proof (cf \cite{B74}) of Mahler's theorem. The first part of the theorem is easy: $f$ is continuous since it is a uniform limit of continuous functions, and if $f(z) = \sum_{n \geq 0} \binom{z}{n} m_n$, then $\vm(f(z)) \geq \inf_{n \geq 0} \vm(m_n)$. The fact that $m_n = (-1)^n \sum_{i=0}^n(-1)^i\binom{n}{i} f(i)$ is a classical exercise, given that $f(k) = \sum_{j=0}^k \binom{k}{j} m_j$ for all $k \geq 0$, and it implies that $\vm(m_n) \geq \inf_{z \in \Zp} \vm(f(z))$ for all $n$. In order to show the converse, it is enough to show that if $f$ is continuous and $m_n(f) = (-1)^n \sum_{i=0}^n(-1)^i\binom{n}{i} f(i)$, then $m_n(f) \to 0$. Indeed, the functions $f$ and $z \mapsto \sum_{n \geq 0} \binom{z}{n} m_n(f)$ are then two continuous functions on $\Zp$ with the same values on $\ZZ_{\geq 0}$, so that they are equal.

We now show that $m_n(f) \to 0$. If $s \geq 0$, there exists $t$ such that if $\vp(x-y) \geq t$ then $\vm(f(x)-f(y)) \geq s$, as $f$ is uniformly continuous. Take $n \geq p^t$ and write $n = qp^t+r$ with $0 \leq r < p^t$ and $q \geq 1$. Writing $i=a+jp^t$, we get 
\[ m_n(f) = \sum_{a=0}^{p^t-1}\sum_{j = 0}^q
(-1)^{n+a+jp^t}\binom{n}{a+jp^t}f(a+jp^t). \]
As we are in characteristic $p$, Lucas' theorem implies that $\binom{n}{a+jp^t} =
\binom{r}{a}\binom{q}{j}$, so that:
\[ m_n(f) = 
\sum_{a=0}^{p^t-1}(-1)^{n+a}\binom{r}{a}
\left(\sum_{j=0}^q (-1)^j\binom{q}{j}f(a+jp^t)\right). \]
As $\left(\sum_{j=0}^q (-1)^j\binom{q}{j} \right) \cdot f(a)=0$, and $\vm(f(a+jp^t)-f(a)) \geq s$ for all $j$, we get that $\vm(m_n(f)) \geq s$ if $n \geq p^t$.
\end{proof}

We now give a characterization of super-H\"older functions in terms of their Mahler expansions.

\begin{prop}
\label{shmahl}
If $f \in C^0(\Zp,M)$, then $f \in \HH^{\lambda,\mu}(\Zp,M)$ if and only if for all $i \geq 0$, we have $\vm(m_n(f)) \geq p^\lambda \cdot p^i +\mu$ whenever $n \geq p^i$.
\end{prop}

\begin{proof}
Take $f \in C^0(\Zp,M)$ such that $\vm(m_n(f)) \geq p^\lambda \cdot p^i +\mu$ whenever $n \geq p^i$. Recall that if $a \in \Zp$ and $i \geq 1$, then for all $j < p^i$ we have $\binom{a}{j} = \binom{a+p^i}{j}$ in $\Fp$. If $z \in \Zp$ and $i \geq 1$, then
\[ f(z+p^i)-f(z) =
\sum_{n \geq 0} m_n(f) \left(\binom{z+p^i}{n}-\binom{z}{n}\right) =
\sum_{n \geq p^i}m_n(f)\left(\binom{z+p^i}{n}-\binom{z}{n}\right). \]
Since $\vm(m_n(f)) \geq p^\lambda \cdot p^i +\mu$ whenever $n \geq p^i$, the formula above implies that
$\vm(f(x+p^i)-f(x)) \geq p^\lambda \cdot p^i +\mu$.
Iterating this, we get that $\vm(f(x+kp^i)-f(x)) \geq p^\lambda \cdot p^i +\mu$ for all $k \in \ZZ_{\geq 0}$. By continuity, this implies that $\vm(f(y)-f(x)) \geq p^\lambda \cdot p^i +\mu$ for all $x,y \in \Zp$ such that $\vp(y-x) \geq i$.

Assume now that $f \in \HH^{\lambda,\mu}(\Zp,M)$. We prove that for all $i \geq 0$ and $n \geq p^i$, we have $\vm(m_n(f)) \geq p^\lambda \cdot p^i +\mu$. Fix $i \geq 0$ and take $a \in \{0,\dots,p^i-1\}$. Define a function $g$ on $\Zp$ by $g(z) = f(a+p^i z)-f(a)$. By hypothesis, we have $\vm(g(z)) \geq p^\lambda \cdot p^i +\mu$ for all $z$. This implies that $\vm(m_n(g)) \geq p^\lambda \cdot p^i +\mu$ for all $n$. We now compute $m_n(g)$. We have
\begin{align*}
g(z) &= 
\sum_{n \geq 0}\left(\binom{a+p^iz}{n}-\binom{a}{n}\right)m_n(f) \\
&=
\sum_{n \geq p^i}\left(\binom{a+p^iz}{n}-\binom{a}{n}\right)m_n(f) 
=
\sum_{n \geq p^i}\binom{a+p^iz}{n}m_n(f),
\end{align*}
since $a \leq p^i-1$. If we write $n = t+p^i \ell$, with $0 \leq t \leq p^i-1$ and $\ell \geq 1$, then
$\binom{a+p^i z}{n} = \binom{a}{t}\binom{z}{\ell}$. This implies that
\[ g(z) =\sum_{t=0}^{p^i-1}\sum_{\ell \geq 1}\binom{a}{t}\binom{z}{\ell}m_{t+p^i\ell}(f), \]
which gives $m_\ell(g) =
\sum_{t=0}^{p^i-1}\binom{a}{t}m_{t+p^i\ell}(f)$
for all $\ell \geq 1$. This now implies that \[ \vm\left(\sum_{t=0}^{p^i-1}\binom{a}{t}m_{t+p^i\ell}(f)\right) \geq p^\lambda \cdot p^i +\mu \] 
for all $\ell \geq 1$ and $a \in \{0,\dots,p^i-1\}$.
The matrix $\left(\binom{a}{t}\right)_{0 \leq a,t \leq p^i-1}$ is unipotent with integral coefficients. Hence for a given $\ell \geq 1$, the above inequality implies that $\vm(m_{a+p^i\ell}(f)) \geq p^\lambda \cdot p^i +\mu$ for all $a \in \{0,\dots,p^i-1\}$. Writing $n \geq p^i$ as $n=a+p^i \ell$, we get $\vm(m_n(f)) \geq p^\lambda \cdot p^i +\mu$ for all $n \geq p^i$.
\end{proof}

\begin{rema}
\label{mahlvssh}
Let $\WW^{\lambda,\mu}(\Zp,M)$ denote the set of $f \in C^0(\Zp,M)$ such that $\vm(m_n(f)) \geq p^\lambda n +\mu$ for all $n \geq 0$. 

Prop \ref{shmahl} implies that $\WW^{\lambda,\mu}(\Zp,M) \subset \HH^{\lambda,\mu}(\Zp,M) \subset \WW^{\lambda-1,\mu}(\Zp,M)$.
\end{rema}

Prop \ref{shmahl} and remark \ref{mahlvssh} strengthen the analogy between
our definition of super-H\"older functions and the classical
theory of locally analytic functions. Indeed, if $f : \Zp \to
\Qp$ is a continuous function, and if $f(z) = \sum_{n \geq 0}
\binom{z}{n} m_n(f)$ is its Mahler expansion, then by a result of Amice
(\cite{A64}, see corollary I.4.8 of \cite{C10}), $f$ is locally analytic
if and only if there exists $\lambda,\mu \in \RR$ such that $\vp(m_n(f))
\geq p^\lambda \cdot n + \mu$ for all $n \geq 0$.

\begin{rema}
\label{gulotta}
Daniel Gulotta pointed out to us that Gulotta (in \S 3 of \cite{G19}), as well as Johansson and Newton (in \S 3.2 \cite{JN19}), had defined a generalization of locally analytic functions, for functions valued in certain general Tate $\Zp$-algebra. When $p=0$ in the algebra, their definition is equivalent to our definition of super-H\"older functions.
\end{rema}

\section{Decompletion of cyclotomic perfectoid fields}
\label{secdecomp}

Let $\efont^+  = \kf \dcroc{X}$. For $n \geq 0$, let $\efont^+_n  = \kf
\dcroc{X^{1/p^n}}$, so that $\efont^+ = \efont^+_0$. Let $\efont^+_\infty
= \cup_{n \geq 0} \efont^+_n$ and let $\et^+$ be the $X$-adic completion
of $\efont^+_\infty$. 
We denote by $\efont$, $\efont_n$, $\efont_\infty$, $\et$ the fields
$\efont^+[1/X]$, $\efont_n^+[1/X]$, $\efont_\infty^+[1/X]$, $\et^+[1/X]$
respectively.
The ring $\et^+$ is the ring of integers of the
field $\et = \et^+[1/X]$. If $\kf$ is perfect, then $\et$ is perfectoid. 

\subsection{The action of $\Zp^\times$}
\label{secgam}

The group $\Zp^\times$ acts continuously by isometries on each
$\efont^+_n$ by the formula $a \cdot X^{1/p^n} = (1+X^{1/p^n})^a-1$. This
action is compatible when $n$ varies, extends to the fields $\efont_n$,
and extends by continuity to $\et^+$ and $\et$.

\begin{rema}
\label{perfcyc}
If $\kf=\Fp$, then $\et$ is the tilt of $\widehat{\Qp(\mu_{p^\infty})}$ (see \S\ref{secfon} for more details). The group $\Gamma=\Gal(\Qp(\mu_{p^\infty})/\Qp)$ is isomorphic to $\Zp^\times$ via the cyclotomic character $\chi_{\cyc}$, and acts on $\et$ by $g(f) = \chi_{\cyc}(g) \cdot f$.
\end{rema}

If $k \geq 1$ (or $k \geq 2$ if $p=2$), let $\Gamma_k = 1+p^k \Zp$. The natural coordinate on $\Gamma_k$ is given by $1+p^k a \mapsto \log_p(1+p^k a) / p^k$. It differs from the coordinate $1+p^k a \mapsto a$ (which is not a group homomorphism) by an isometry. By remark \ref{isom}, the definition of $(\et^+)^{\Gamma_k \hyphen \sh,\lambda,\mu}$ does not depend on the choice of one of those coordinates, and we use $1+p^k a \mapsto a$.

\begin{prop}
\label{etnsh}
We have $\efont^+_n  = (\efont^+_n)^{\Gamma_k \hyphen \sh,k-n,0}$.
\end{prop}

\begin{proof}
We have $(1+X^{1/p^n})^{1+p^{k+i}b} = (1+X^{1/p^n}) \cdot (1+X^{p^{k+i-n}})^b$, so that
\[ \vx((1+X^{1/p^n})^{1+p^{k+i}b} - (1+X^{1/p^n}))   \geq p^{k-n} \cdot p^i. \]
This implies that $X^{1/p^n} \in (\efont^+_n)^{\Gamma_k \hyphen \sh,k-n,0}$. The claim now follows from prop \ref{propmsh} and lemma \ref{cplmsh}.
\end{proof}

Taking $n=0$ in prop \ref{etnsh}, we find that $\kf \dcroc{X} = \kf \dcroc{X}^{\Gamma_k \hyphen \sh,k}$. Let $\efont=\efont^+[1/X]$.

\begin{coro}
\label{egsk}
We have $\efont = \efont^{\Gamma_k \hyphen \sh,k}$.
\end{coro}

\begin{proof}
This follows from prop \ref{etnsh} and prop \ref{propmsh}.
\end{proof}

\begin{prop}
\label{levelzer}
If $\eps>0$, then $\kf \dcroc{X}^{\Gamma_k \hyphen \sh,k + \eps} \subset \kf \dcroc{X^p}$. 
\end{prop}

\begin{proof}
Take $f(X) \in \kf \dcroc{X}$. There is a power series $h(Y,Z) \in \kf \dcroc{Y,Z}$ such that \[ f(Y+Z) = f(Y) + Z \cdot f'(Y) + Z^2 \cdot h(Y,Z). \] If $m \geq 0$,  this implies that
\begin{align*} 
f((1+X)^{1+p^m}-1)  & = f(X+X^{p^m}(1+X)) \\
& = f(X) +  X^{p^m}(1+X) \cdot f'(X) + \bigO(X^{2p^m}). 
\end{align*}
If $f(X) \notin \kf \dcroc{X^p}$, then $f'(X) \neq 0$. Let $\mu=\vx(f'(X))$. The above computations imply that $\vx((1+p^{i+k}) \cdot f(X) - f(X)) = p^{i+k}+\mu$ for $i \gg 0$. This implies the claim.
\end{proof}

\begin{coro}
\label{etshlevel}
We have $(\efont^+_\infty)^{\Gamma_k \hyphen \sh, k-n} = \efont^+_n$.
\end{coro}

\begin{proof}
Take $f(X^{1/p^m}) \in (\efont^+_\infty)^{\Gamma_k \hyphen \sh, k-n}$
where $f(X) \in \kf \dcroc{X}$. Since $\vx(h^p) = p \cdot \vx(h)$ for all
$h \in \et^+$, we have $f^{p^m}(X) \in (\efont^+_\infty)^{\Gamma_k
\hyphen \sh, k+m-n}$, where $f^{p^m}(X) \in \kf \dcroc{X}$ is 
$f^{p^m}(X) = f(X^{1/p^m})^{p^m}$. If $m \geq n+1$, then prop
\ref{levelzer} implies that $f^{p^m}(X) \in \kf \dcroc{X^p}$, so that
$f(X) = g(X^p)$, and $f(X^{1/p^m}) = g(X^{1/p^{m-1}})$. This implies the
claim.
\end{proof}

\subsection{Tate traces}
\label{sectate}

We recall some constructions of Colmez (see \S 8.2 of \cite{C08}). For $m \geq 0$ let
$I_m = p^{-m}\ZZ\cap [0,1)$, and let $I = \cup_m I_m$. Note that if $i
\in I_m$, then $(1+X)^i \in \efont^+_m$.

\begin{lemm}
The elements $(1+X)^i$, $i\in I_m$, form a basis of $\efont^+_m$ over $\efont^+_0$.
\end{lemm}

\begin{proof}
See lemma 8.2 of \cite{C08}. Colmez works with $\kf = \Fp$, but the proofs are the same with arbitrary coefficients.
\end{proof}

\begin{prop}
\label{colmet+}
Any $f \in \et^+$ can be written uniquely as $\sum_{i\in
I}(1+X)^ia_i(f)$, with $a_i(f) \in \efont^+_0$, and $a_i(f) \to 0$.
Moreover, $\vx(f)-1 < \inf_{i\in I} \vx(a_i(f)) \leq \vx(f)$. 
\end{prop}

\begin{proof}
See props 4.10 and 8.3 of \cite{C08}.
\end{proof}

In particular, for all $i \in I$, the map $\et^+ \to \efont^+_0$, given by $f \mapsto a_i(f)$ is continuous.

\begin{prop}
\label{colmtn}
There exists a family $\{ T_n \}_{n \geq 0}$ of continuous maps $T_n : \et^+ \to \efont^+_n$ satisfying the following properties:
\begin{enumerate}
\item The restriction of $T_n$ to $\efont^+_n$ is the identity map.
\item We have $T_n(f) \to f$ as $n \to +\infty$.
\item We have $\vx(T_n(f)) \geq \vx(f)-1$ for all $n$.
\item Each $T_n$ is $\Zp^\times$-equivariant.
\end{enumerate}
\end{prop}

\begin{proof}
If $f = \sum_{i\in I}(1+X)^ia_i(f)$, let $T_n(f) = \sum_{i \in I_n}(1+X)^ia_i(f)$.
With this definition, the first property is immediate. The second and third
one follow from prop \ref{colmet+}.

For the last one, observe that if $i \in I$ and $g \in \Zp^\times$, then $g
\cdot (1+X)^i =(1+X)^{ig}$ so $g \cdot (1+X)^i$ can be
written uniquely as $(1+X)^{\sigma_g(i)}u_{i,g}(X)$ with $\sigma_g(i) \in
I$ and $u_{i,g}(X) \in \efont^+_0$. The map $\sigma_g$ induces a bijection
from $I_m$ to itself for all $m$. Take $f \in \et^+$, and write $f = \sum_{i\in I}(1+X)^ia_i(f)$. We have
$g \cdot f = \sum_{i\in I}(1+X)^{\sigma_g(i)}u_{i,g}(X)(g\cdot a_i(f))$,
so that $T_n(g \cdot f) = \sum_{i\in I_n}(1+X)^{\sigma_g(i)}u_{i,g}(X)(g\cdot a_i(f)) = g \cdot T_n(f)$.
\end{proof}

\subsection{Decompletion of $\et$}
\label{subdec}

We now prove that $\et^{\sh} = \efont_\infty$. More precisely, we have the following result.

\begin{theo}
\label{shdecet}
We have $\et^{\Gamma_k\hyphen\sh, k-m} = \efont_m$ for all $m \geq 0$, and $\et^{\sh} = \efont_\infty$.
\end{theo}

\begin{prop}
\label{tnsh}
If $f \in (\et^+)^{\Gamma_k\hyphen \sh, \lambda,\mu}$, then $T_n(f) \in  (\efont^+_n)^{\Gamma_k \hyphen \sh, \lambda,\mu-1}$.
\end{prop}

\begin{proof}
If $g \in \Gamma_k$, then $g(T_n(f)) - T_n(f) = T_n(g(f)-f)$ so that \[ \vx(g(T_n(f)) - T_n(f)) = \vx (T_n(g(f)-f)) \geq \vx(g(f)-f)-1 \]
by prop \ref{colmtn}. This implies the claim.
\end{proof}

\begin{proof}[Proof of theorem \ref{shdecet}]
Take $f \in (\et^+)^{\Gamma_k \hyphen \sh, k-m}$. By prop \ref{tnsh}, we have $T_n(f) \in  (\efont^+_n)^{\Gamma_k \hyphen \sh, k-m}$ for all $n \geq 0$. By coro \ref{etshlevel}, $T_n(f) \in \efont^+_m$ for all $n$. Since $T_n(f) \to f$ as $n \to +\infty$, we have $f \in \efont^+_m$.  

Hence $(\et^+)^{\Gamma_k\hyphen\sh, k-m} = \efont^+_m$, and this implies the theorem by prop \ref{propmsh}.
\end{proof}

\section{Applications}
\label{secappli}

We now give several applications of the fact that $\et^{\sh} = \efont_\infty$.

\subsection{The perfectoid commutant of $\Aut(\Gm)$}
\label{secperfcom}

In this section, we assume that $\kf=\Fp$. If $a \in \Zp^\times$, let $\gamma_a(X) = (1+X)^a-1 \in \Fp \dcroc{X}$. Note that if $f \in \et$, then $a \cdot f = f \circ \gamma_a$.  If $u \in \et^+$ is such that $\vx(u)>0$, the series $\gamma_a \circ u$ converges in $\et^+$. If $u = \gamma_b(X^{p^n})$ for some $b \in \Zp^\times$ and $n \in \ZZ$, then $u \circ  \gamma_a = \gamma_a \circ u$ for all $a \in \Zp^\times$.

\begin{theo}
\label{gmcom}
If $u \in \et^+$ is such that $\vx(u) > 0$ and $u \circ  \gamma_a = \gamma_a \circ u$ for all $a \in \Zp^\times$, then there exists $b \in \Zp^\times$ and $n \in \ZZ$ such that $u(X) = \gamma_b(X^{p^n})$.
\end{theo}

Recall that a power series $f(X) \in \Fp\dcroc{X}$ is separable if $f'(X)
\neq 0$. If $f(X) \in X \cdot \Fp\dcroc{X}$, we say that $f$ is
invertible if $f'(0) \in \Fp^\times$, which is equivalent to $f$ being
invertible for composition (denoted by $\circ$). We say that $w(X) \in X
\cdot \Fp\dcroc{X}$ is nontorsion if $w^{\circ n}(X) \neq X$ for all $n
\geq 1$. The following is a reformulation of lemma 6.2 of \cite{L94}.

\begin{lemm}
\label{lubnarch}
Let $w(X) \in X + X^2\cdot \Fp\dcroc{X}$ be an invertible nontorsion series, and let $f(X) \in X \cdot \Fp\dcroc{X}$ be a separable power series. If $w \circ f = f \circ w$, then $f$ is invertible.
\end{lemm}

\begin{lemm}
\label{locansh}
If $u \in \et^+$ is such that $\vx(u) > 0$ and $u \circ  \gamma_a = \gamma_a \circ u$ for all $a \in \Zp^\times$, then $u \in (\et^+)^{\sh}$.
\end{lemm}

\begin{proof}
The group $\Zp^\times$ acts on $\et^+$ by $a \cdot u = u \circ \gamma_a$, so we need to check that the function $a \mapsto \gamma_a \circ u$ is super-H\"older. This is clear since $\gamma_a(u) = \sum_{n \geq 1} \binom{a}{n} u^n$ and $\vx(u)>0$.
\end{proof}

\begin{proof}[Proof of theorem \ref{gmcom}] 
Take $u \in \et^+$ such that $\vx(u) > 0$ and $u \circ  \gamma_a = \gamma_a \circ u$ for all $a \in \Zp^\times$. By lemma \ref{locansh} and theorem \ref{shdecet}, there exists $m \geq 0$ such that $u \in \efont^+_m$. Hence there is an $n \in \ZZ$ such that $f(X) = u(X^{1/p^n})$ belongs to $X \cdot \Fp\dcroc{X}$ and is separable. Take $g \in 1+p \Zp$ such that $g$ is nontorsion, and let $w(X) = \gamma_g(X)$ so that $u \circ w = w \circ u$. We also have $f \circ w = w \circ f$. By lemma \ref{lubnarch}, $f$ is invertible. Since $f \circ  \gamma_a = \gamma_a \circ f$ for all $a \in \Zp^\times$, theorem 6 of \cite{LS07} implies that $f \in \Aut(\Gm)$. Hence there exists $b \in \Zp^\times$ such that $f(X)=\gamma_b(X)$. This implies the theorem.
\end{proof}

\subsection{Decompletion of $(\phi,\Gamma)$-modules}
\label{secpgm}

Let $\Gamma_k = 1+p^k \Zp$ with $k \geq 1$, as in \S\ref{secgam}. Let $M$ be a
finite-dimensional $\efont$-vector space with a continuous semi-linear action of
$\Gamma_k$.

\begin{prop}
\label{gamlat}
There is an $\efont^+$-lattice in $M$ that is stable under $\Gamma_k$.
\end{prop}

\begin{proof}
Choose any lattice $M_0^+$ of $M$. The map $\pi : \Gamma_k \times M \to M$ is continuous, so there is an open subgroup $H$ of $\Gamma_k$ and an $n \geq 0$ such that $\pi^{-1}(M_0^+)$ contains $H \times X^n M_0^+$. In particular, $h(m) \in X^{-n} M_0^+$ for all $h \in H$ and $m \in M_0^+$. Since $H$ is open in the compact group $\Gamma_k$, it is of finite index, and there exists $d \geq n$ such that $g(m) \subset X^{-d} M_0^+$ for all $g \in \Gamma_k$ and $m \in M_0^+$. The space $M^+ = \sum_{g \in \Gamma_k} g(M_0^+)$ is an $\efont^+$-module such that $M_0^+ \subset M^+ \subset X^{-d} M_0^+$, so that $M^+$ is a lattice of $M$. It is clearly stable under $\Gamma_k$.
\end{proof}

Choosing such an $\efont^+$-lattice in $M$
defines a valuation $\vm$ on $M$, such that $\Gamma_k$ acts on $M$ by isometries.
We make such a choice, and we can therefore define
$M^{\sh}$ and $M^{\Gamma_k\hyphen\sh,\lambda}$ as in definition \ref{defshvec}. We say
that the action of $\Gamma_k$ on $M$ is super-H\"older if $M=M^{\sh}$.

\begin{lemm}
\label{mshsub}
The space $M^{\Gamma_k\hyphen\sh,\lambda}$ does not depend on the choice of $\Gamma_k$-stable lattice of $M$.  If $\lambda \leq k$ then $M^{\Gamma_k\hyphen\sh,\lambda}$ is sub-$\efont$-vector space
of $M$.
\end{lemm}

\begin{proof}
The first assertion results from the fact that if we choose two
$\efont^+$-lattices $M_1^+$ and $M_2^+$ in $M$, then there exists a
constant $C$ such that $|{\val_1-\val_2} | \leq C$.

Next, recall that by coro \ref{egsk}, $\efont =
\efont^{\Gamma_k\hyphen\sh,k}$. If $m \in M^{\sh,\lambda}$,  $f \in
\efont$, and $g \in \Gamma_k$, then $g(fm)-fm = g(f)(g(m)-m) + (g(f)-f)m$,
so that $fm \in M^{\sh,\lambda}$ by lemma \ref{mshalt}.
\end{proof}

Lemma \ref{mshsub} implies that $M^{\sh}$ is a sub-$\efont$-vector space of $M$.  We say that a basis of $M$ is good if it generates a lattice that is stable under $\Gamma_k$.

\begin{prop}
\label{mshbas}
Take $\lambda \leq k$ and fix a good basis of $M$. We have $M=M^{\Gamma_k\hyphen\sh,\lambda}$ if and only if the
map $\Gamma_k \to \M_n(\efont^+)$, given by $g \mapsto \Mat(g)$, is in $\HH^\lambda(\Gamma_k,\M_n(\efont^+))$. 
\end{prop}

\begin{proof}
We fix a good basis $(m_1,\dots,m_n)$ of $M$, and work with the corresponding valuation $\vm$ on $M$. By lemma \ref{mshsub}, we have $M=M^{\Gamma_k\hyphen\sh,\lambda}$ if and only if $m_j \in M^{\Gamma_k\hyphen\sh,\lambda}$ for all $j$. We have $g \cdot m_j = \sum_{i=1}^n\Mat(g)_{i,j} m_i$ by definition of $\Mat(g)$. Hence if $g,h \in \Gamma_k$, then $g \cdot m_j - h \cdot m_j = \sum_{i=1}^n (\Mat(g)_{i,j} - \Mat(h)_{i,j}) m_i$. This implies that if $\ell \geq 0$ and $\mu \in \RR$, then $\vm(g \cdot m_j - h \cdot m_j) \geq p^{\lambda+\ell} + \mu$ if and only if $\vx(\Mat(g)-\Mat(h)) \geq p^{\lambda+\ell} + \mu$. This implies the claim.
\end{proof}

If $M$ is a finite-dimensional $\efont$-vector space with a semi-linear
action of $\Gamma_k$, then $\et \otimes_{\efont} M$ is a finite-dimensional
$\et$-vector space with a semi-linear action of $\Gamma_k$. If $M$ is super-H\"older, there exists $m_0 = m_0(M) \geq 0$ such that $M=M^{\Gamma_k\hyphen\sh,k-m_0}$

\begin{prop}
\label{vecbash}
If $M$ is super-H\"older and $m \geq m_0(M)$, then $(\et \otimes_{\efont} M)^{\Gamma_k\hyphen\sh,k-m} = \efont_m \otimes_{\efont} M$.
\end{prop}

\begin{proof}
By the same argument as in the proof of lemma \ref{mshsub}, we see that for $m \geq m_0$, $(\et \otimes_{\efont} M)^{\Gamma_k\hyphen\sh,k-m}$ is a sub-$\efont_m$-vector space of $\et \otimes_{\efont} M$. The space $(\et \otimes_{\efont} M)^{\Gamma_k\hyphen\sh,k-m}$ contains $M$, and therefore also $\efont_m \otimes_{\efont} M$. This proves one inclusion. 

We now prove that $(\et \otimes_{\efont} M)^{\Gamma_k\hyphen\sh,k-m} \subset \efont_m \otimes_{\efont} M$. Fix a good basis $(m_1,\dots,m_n)$ of $M$, the corresponding valuation $\vm$ on $\et \otimes_{\efont} M$, and $m \geq m_0$. Take $x = \sum_{i=1}^n x_i m_i \in \et \otimes_{\efont} M$ and write $g(x) = \sum_{i=1}^n f_i(g) m_i$. We have $x \in (\et \otimes_{\efont} M)^{\Gamma_k\hyphen\sh,k-m}$ if and only if $f_i \in \HH^{k-m}(\Gamma_k,\et)$ for all $i$. In addition, $g(x) = \sum_{i,j} g(x_i) \Mat(g)_{j,i} m_j$. Hence $f_j : g \mapsto \sum_{i=1}^n g(x_i) \Mat(g)_{j,i}$ belongs to $\HH^{k-m}(\Gamma_k,\et)$ for all $j$. We have $g(x_\ell) = \sum_{j=1}^n f_j(g) (\Mat(g)^{-1})_{\ell,j}$. By props \ref{mshbas} and \ref{hrng}, $[g \mapsto g(x_\ell) ] \in \HH^{k-m}(\Gamma_k,\et)$ and therefore $x_\ell \in \et^{\Gamma_k\hyphen\sh,k-m} = \efont_m$ for all $\ell$.
\end{proof}

\begin{coro}
\label{vecbashcor}
If $M$ is super-H\"older, then $(\et \otimes_{\efont} M)^{\sh} =
\efont_\infty \otimes_{\efont} M$.
\end{coro}

The field $\efont = \kf \dpar{X}$ is equipped with its action of $\Zp^\times$ and with the $\kf$-linear Frobenius map $\phi$ given by $\phi(f)(X) = f(X^p)$. Let $\Gamma = \Gamma_k$ with $k \geq 1$. A $(\phi,\Gamma)$-module $\dfont$ over $\efont$ is a finite-dimensional $\efont$-vector space, endowed with commuting, semi-linear actions of $\phi$ and $\Gamma$, such that the action of $\Gamma$ is continuous and such that $\Mat(\phi)$ is invertible (in any basis of $\dfont$).

\begin{prop}
\label{phigsh}
If $\dfont$ is a $(\phi,\Gamma)$-module over $\efont$, then $\dfont=\dfont^{\Gamma_k\hyphen\sh,k}$.
\end{prop}

\begin{lemm}
\label{hmnep}
If $\ell \geq 1$ and $\lambda,\mu \in \RR$, then $\HH^{\lambda,\mu}(\Gamma_\ell,\M_n(\efont^+))$ is a ring, that is stable under $\phi$.
\end{lemm}

\begin{proof}
The first claim follows from prop \ref{hrng}. The second one follows from the fact that if $M \in \M_n(\efont^+)$, then $\vx(\phi(M)) \geq \vx(M)$.
\end{proof}

\begin{proof}[Proof of of proposition \ref{phigsh}]
Choose a good basis $(d_1,\dots,d_n)$ of $\dfont$. We can replace $(d_1,\dots,d_n)$ by $(X^s d_1,\dots, X^s d_n)$ for some $s \geq 0$, and assume that $P = \Mat(\phi) \in \M_n(\efont^+)$. 
Take $r \geq 1$ such that $X^r P^{-1} \in X \M_n(\efont^+)$. Let $G_g$ be the matrix of $g \in \Gamma$. By
continuity of the map $\Gamma \to \GL_n(\efont^+)$, $g \mapsto G_g$, there
exists $\ell \geq k$ such that for all $g \in \Gamma_\ell$, we have $\vx(G_g-\Id) \geq
r$. Write $G_g = \Id + X^r H_g$ with $H_g \in \M_n(\efont^+)$.

By definition of $r$, we have $X^r g(P)^{-1} \in X \M_n(\efont^+)$, so
that if $Q_g = X^{r(p-1)}g(P)^{-1}$, then $Q_g \in X\M_n(\efont^+)$. The
commutation relation between $\phi$ and $\Gamma_\ell$ gives $P\phi(G_g)
= G_g g(P)$ for all $g \in \Gamma_\ell$. 
Therefore, $P\phi(\Id+X^rH_g) = (\Id+X^rH_g)g(P)$, so that
\[Pg(P)^{-1}-\Id = X^r(H_g-P\phi(H_g)Q_g). \]
This implies that $Pg(P)^{-1}-\Id \in X^r\M_n(\efont^+)$.
Let \[ f(g) = H_g-P\phi(H_g)Q_g = X^{-r}\left(Pg(P)^{-1}-\Id\right). \] 
Recall that $Q_g,  f(g) \in \M_n(\efont^+)$ for all $g \in \Gamma_\ell$, and that 
(compare with (4) of prop \ref{hrng})
\[ Q_g = X^{r(p-1)} g(P)^{-1} = X^{r(p-1)} g(\tco(P)) g(\det(P)^{-1}) \]
and 
\[  f(g) = X^{-r}\left(P g(\tco(P)) g(\det(P)^{-1})-\Id\right). \]
By props \ref{propmsh} and \ref{etnsh}, and lemma \ref{hmnep}, there
exists $\mu \in \RR$ such that $g \mapsto Q_g$ and $g \mapsto f(g)$ 
belong to $\HH^{\ell,\mu}(\Gamma_\ell,\M_n(\efont^+))$.

Let $f_0 = f$ and for $i \geq 1$,  let $f_i : \Gamma_\ell \to
\M_n(\efont^+)$ be the function \[ g \mapsto
P\phi(P)\cdots\phi^{i-1}(P) \cdot \phi^i(f(g)) \cdot
\phi^{i-1}(Q_g)\cdots\phi(Q_g)Q_g. \] 

Since $P \in \M_n(\efont^+)$, lemma \ref{hmnep} implies that $f_i \in
\HH^{\ell,\mu}(\Gamma_\ell,\M_n(\efont^+))$. In addition, $\vx(Q_g) \geq
1$, so that $\vx(\phi^{i-1}(Q_g)\cdots\phi(Q_g)Q_g) \geq (p^i-1)/(p-1)$.
Hence $\sum_{i \geq 0} f_i$ converges in
$\HH^{\ell,\mu}(\Gamma_\ell,\M_n(\efont^+))$, and we let $T(f)$ be its
limit. 

We have $T(f)(g) = H_g$. This implies that $g \mapsto H_g$ belongs to
$\HH^{\ell,\mu}(\Gamma_\ell,\M_n(\efont^+))$, and hence so does $g
\mapsto G_g = \Id + X^r H_g$. 

We therefore have $\dfont=\dfont^{\Gamma_\ell\hyphen\sh,\ell}$, so that
$\dfont=\dfont^{\Gamma_k\hyphen\sh,k}$ by lemma \ref{shsbgr}.
\end{proof}

\begin{coro}
\label{decompgm}
If $\dfont$ is a $(\phi,\Gamma)$-module over $\efont$, then $(\et \otimes_{\efont} \dfont)^{\Gamma_k\hyphen\sh,k-m} = \efont_m \otimes_{\efont} \dfont$ for $m \geq 0$.
\end{coro}

We now prove the following result, which generalizes prop \ref{phigsh}. Note that the underlying constants are not as good as in the case of a $(\phi,\Gamma)$-module.

\begin{prop}
\label{gamodsh}
If $M$ is a finite-dimensional $\efont$-vector space with a continuous semi-linear action of $\Gamma_k$, then $M=M^{\sh}$. 
\end{prop}

\begin{proof}
Choose a good basis of $M$. Let $f(g)$ denote the matrix of $g \in \Gamma$ in this basis. If $\ell \geq 1$, there exists $k \geq \ell+1$ such that $f(g) \in \Id + X^{p^\ell} \M_n(\efont^+)$ for all $g \in 1+p^k \Zp$. Write $f(g) = \Id + X^{p^\ell} H$. The cocycle formula gives \[ f(g^p) = (\Id + X^{p^\ell} H) (\Id + g(X^{p^\ell} H))  \cdots (\Id + g^{p-1}(X^{p^\ell} H)). \]
Prop \ref{etnsh}, with $n=0$, implies that $g^m(X^{p^\ell} H) \equiv X^{p^\ell} H \bmod{X^{p^k}}$ for all $0 \leq m \leq p-1$. Hence $f(g^p) \equiv (\Id + X^{p^\ell} H)^p \bmod{X^{p^k}}$.  This implies that $f(g^p) \equiv \Id + X^{p^{\ell+1}} H^p \bmod{X^{p^k}}$ so that $f(g^p) = \Id \bmod{X^{p^{\ell+1}}}$ since $k \geq \ell+1$.

Since $(1+p^k \Zp)^p = 1+p^{k+1} \Zp$, the above computation implies by induction on $i$ that $f(1+p^{k+i} \Zp) \subset \Id + X^{p^{\ell+i}} \M_n(\efont^+)$ for all $i \geq 0$. 

This implies that $M=M^{\Gamma_k\hyphen\sh,\ell,0}$ by lemma \ref{mshalt}.
\end{proof}

\begin{coro}
\label{senfin}
Let $N$ be an $\efont$-vector space, with a compatible valuation and a semi-linear action of $\Gamma_k$ by isometries. Let $N^{\fin}$ denote the set of $x \in N$ that belong to a finite dimensional $\efont$-vector space stable under $\Gamma_k$, in analogy with classical Sen theory. 

Prop \ref{gamodsh} implies that $N^{\fin} \subset N^{\sh}$. In particular, if $N=\et$, then $\et^{\fin} = \et^{\sh} = \efont_\infty$.
\end{coro}

\subsection{The field of norms}
\label{secfon}

Let $K$ be a finite extension of $\Qp$. Let $K_n = K (\mu_{p^n})$ and let $K_\infty = \cup_{n \geq 0} K_n$. The field of norms of the extension $K(\mu_{p^\infty})/K$ is defined and studied in \cite{W83}. It is the set of sequences $\{x_n\}_{n \geq 0}$ where $x_n \in K_n$ and $\Nm_{K_{n+1}/K_n}(x_{n+1}) = x_n$ for all $n \geq 0$. This set has a natural structure of a field of characteristic $p$ whose residue field is that of $K_\infty$ (\S 2.1 of ibid), which we denote by $\efont_K$. If $K=\Qp$, then $\efont_{\Qp} = \Fp \dpar{X}$, where $X=\{x_n\}_{n \geq 0}$ with $x_n = 1-\zeta_{p^n}$ for $n \geq 1$. When $K$ is a finite extension of $\Qp$, $\efont_K$ is a finite separable extension of $\efont_{\Qp}$ of degree $[K_\infty:(\Qp)_\infty]$ (\S 3.1 of ibid). 

Let $\Gamma_K = \Gal( K_\infty  / K)$, so that $\Gamma_K$ is isomorphic to an open subgroup of $\Zp^\times$ via the cyclotomic character $\chi_{\cyc}$. The group $\Gamma_K$ acts naturally on $\efont_K$, and if $g \in \Gamma_K$, then $g(X) = (1+X)^{\chi_{\cyc}(g)}-1$. Let $\phi : \efont_K \to \efont_K$ denote the map $y \mapsto y^p$. Let $\et_K$ denote the $X$-adic completion of $\cup_{n \geq 0} \phi^{-n}(\efont_K)$. In particular, $\et_{\Qp} = \et$ in the notation of \S\ref{secdecomp}, and $\et_K$ is the tilt of $\widehat{K}_\infty$ (\S 4.3 of ibid and \S 3 of \cite{S12}).

\begin{lemm}
\label{ffpgm}
We have $\phi^{-n}(\efont_K) = \efont_n \otimes_{\efont} \efont_K$ for all $n$, and $\et_K = \et \otimes_{\efont} \efont_K$. 
\end{lemm}

\begin{proof}
The extensions $\efont_n/\efont$ and $\efont_K/\efont$ are linearly disjoint since the first is purely inseparable and the second is separable. By comparing degrees, we get the first claim. It implies that $ \et \otimes_{\efont} \efont_K \to \et_K$ is surjective, and the second claim follows, since $[\et_K:\et] = [\efont_K:\efont] = [K_\infty:(\Qp)_\infty]$.
\end{proof}

\begin{coro}
\label{ffsh}
We have $\et_K^{\sh} = \cup_{n \geq 0} \phi^{-n}(\efont_K)$.
\end{coro}

\begin{proof}
This follows from lemma \ref{ffpgm} and coro \ref{decompgm}, as $\efont_K$ is a $(\phi,\Gamma_K)$-module over $\efont$, and $\cup_{n \geq 0} \phi^{-n}(\efont_K) = \efont_\infty \otimes_\efont \efont_K$.
\end{proof}

\begin{rema}
\label{bcla}
In characteristic zero, $\hat{K}_\infty$ is a $p$-adic Banach representation of $\Gamma_K$, and by theorem 3.2 of \cite{BC16}, $K_\infty$ is the space $\hat{K}_\infty^{\mathrm{la}}$ of locally analytic vectors in $\hat{K}_\infty$.
\end{rema}

\subsection{The $p$-adic local Langlands correspondence}
\label{secllp}

We now prove a result that suggests that the theory of super-H\"older
vectors could have some applications to the $p$-adic local Langlands
correspondence. In order to avoid too many technicalities, we consider
only the simplest example. Recall that if $f \in \efont^+$,  there exist
$f_0,\hdots,f_{p-1} \in \efont^+$ such that $f = \sum_{i=0}^{p-1}
\phi(f_i) (1+X)^i$. We define $\psi(f)=f_0$. The map $\psi : \efont^+ \to
\efont^+$ has the following properties: $\psi(f \phi(h)) = h \psi(f)$ if
$f,h \in \efont^+$ and $\psi \circ g = g \circ \psi$ if $g \in
\Zp^\times$.

Let $M = \varprojlim_\psi \efont^+$ be the set of sequences $m = (m_0,m_1,\hdots)$ with $m_i \in \efont^+$ and $\psi(m_{i+1}) = m_i$ for all $i \geq 0$. The space $M$ is endowed with an action of $\Zp^\times$ given by $(g \cdot m)_i = g \cdot m_i$ and the structure of an $\efont^+$-module given by $(f(X) m)_i = \phi^i(f(X))m_i$. Following Colmez, we could extend these structures to an action of the Borel subgroup $\BB$ of $\GL_2(\Qp)$ on $M$, and this idea is an important step in the construction of the $p$-adic local Langlands correspondence. The representation $M$ is then the dual of most of the restriction to $\BB$ of a parabolic induction. However, we don't use this here. 

Let $\vx$ be the $X$-adic valuation on $M$: $\vx(m)$ is the max of the $n \geq 0$ such that $m \in X^n M$. The space $M$ is separated and complete for the $X$-adic topology, although this is not the natural topology on $M$ (the natural topology is induced by the product topology $\varprojlim_\psi \efont^+ \subset \prod \efont^+$. 
The action of $\Zp^\times$ on $M$ is not continuous for the $X$-adic topology: $M \neq M^{\cont}$ in the notation of remark \ref{mcont}).

We have an injection $i : \efont^+ \to M$, given by $i(f) = (f,\phi(f),\phi^2(f),\hdots)$.

\begin{prop}
\label{llpsh}
We have $M^{\Gamma_k\hyphen\sh,k} = i(\efont^+)$.
\end{prop}

\begin{proof}
Recall that if $m \in M$ and $f(X) \in \efont$, then $(f(X) m)_j = \phi^j(f(X))m_j$ for all $j \geq 0$. We have $\vx(\phi^j(f(X))) = p^j \vx(f(X))$. In particular, if $m \in M^{\Gamma_k\hyphen\sh,k}$, then $m_j \in (\efont^+) ^{\Gamma_k\hyphen\sh,k+j}$. The results of \S \ref{secgam} imply that $m_j \in \phi^j(\efont^+)$. If $m_j = \phi^j(f_j)$, the $\psi$-compatibility implies that $f_j=f_0$ for all $j \geq 0$. This implies the claim.
\end{proof}

A generalization of prop \ref{llpsh} to representations of $\BB$ obtained from $(\phi,\Gamma)$-modules using Colmez' construction shows that using the theory of super-H\"older vectors, we can recover the $(\phi,\Gamma)$-module giving rise to such a representation of $\BB$. One of the main results of \cite{BV} is that every infinite dimensional smooth irreducible $\kf$-linear representation of $\BB$ having a central character comes from a $(\phi,\Gamma)$-module by Colmez' construction. Is it possible to reprove this result using super-H\"older vectors?

\providecommand{\bysame}{\leavevmode ---\ }
\providecommand{\og}{``}
\providecommand{\fg}{''}
\providecommand{\smfandname}{\&}
\providecommand{\smfedsname}{\'eds.}
\providecommand{\smfedname}{\'ed.}
\providecommand{\smfmastersthesisname}{M\'emoire}
\providecommand{\smfphdthesisname}{Th\`ese}

\end{document}